\documentclass[12pt]{article}
\usepackage{amsmath}
\usepackage[left=2cm,right=2cm,top=2cm,bottom=2cm]{geometry}
\usepackage{amssymb,amscd}
\usepackage{amsthm}
\usepackage{graphicx}
\usepackage[normalem]{ulem}
\usepackage[bookmarksnumbered,  plainpages]{hyperref}
\usepackage{multirow}
\usepackage{inputenc}

\newtheoremstyle{case}{}{}{}{}{}{:}{ }{}
\theoremstyle{case}

\newcommand{\be}{\begin{equation}}
\newcommand{\ee}{\end{equation}}
\newcommand{\ben}{\begin{eqnarray*}}
\newcommand{\een}{\end{eqnarray*}}
\newtheorem{examp}{\sc Example}
\newtheorem{remk}{\sc Remark}
\newtheorem{corol}{\sc Corollary}
\newtheorem{lemma}{\sc Lemma}
\newtheorem{theorem}{\sc Theorem}
\newtheorem{defn}{\sc Definition}
\newtheorem{prop}{\sc Proposition}
\newcommand{\bet}{\begin{theorem}}
    \newcommand{\eet}{\end{theorem}}
\newcommand{\bel}{\begin{lemma}}
    \newcommand{\eel}{\end{lemma}}
\newcommand{\bed}{\begin{defn}}
    \newcommand{\eed}{\end{defn}}
\newcommand{\brem}{\begin{remk}}
    \newcommand{\erem}{\end{remk}}
\newcommand{\bex}{\begin{examp}}
    \newcommand{\eex}{\end{examp}}
\newcommand{\bcl}{\begin{corol}}
    \newcommand{\ecl}{\end{corol}}
\newcommand{\bep}{\begin{prop}}
    \newcommand{\eep}{\end{prop}}

\newcommand{\NI}{\noindent}
\newcommand{\bea}{\begin{eqnarray}}
\newcommand{\eea}{\end{eqnarray}}

\newcommand{\vsp}{\vskip 1em}

\begin{document}
\title{\large {\bf{\sc Tracing homotopy path for the solution of nonlinear complementarity  Problem}}}

\author{ A. Dutta$^{a, 1}$, A. K. Das$^{b, 2}$\\
\emph{\small $^{a}$Department of Mathematics, Jadavpur University, Kolkata, 700 032, India}\\	
\emph{\small $^{b}$SQC \& OR Unit, Indian Statistical Institute, Kolkata, 700 108, India}\\
\emph{\small $^{1}$Email: aritradutta001@gmail.com}\\
\emph{\small $^{2}$Email: akdas@isical.ac.in} \\
 }
\date{}

\maketitle

\begin{abstract}

	 \NI In this article, we consider nonlinear complementarity problem. We introduce a new homotopy function for finding the solution of nonlinear complementarity problem through the trajectory . We show that the homotopy path approaching the solution is smooth and bounded .  Numerical example of an oligopoly equilibrium problem is illustrated to show the effectiveness of the proposed algorithm. \\
\NI{\bf Keywords:}  Nonlinear complementarity problem, homotopy continuation method, Oligopoly equilibrium problem. 
\end{abstract}
\section{Introduction}
The nonlinear complementarity problem is identified as an important mathematical programming problem. The idea of nonlinear complementarity problem is based on the concept of linear complementarity problem. For recent study on this problem and applications see \cite{das2017finiteness}, \cite{articlee14}, \cite{bookk1}, \cite{articlee7} and references therein. For details of several matrix classes in complementarity theory, see \cite{articlee1}, \cite{articlee2}, \cite{articlee9}, \cite{articlee17}, \cite{article1}, \cite{mohan2001more}, \cite{article12}, \cite{article07}, \cite{dutta2022column} and references cited therein. The problem of computing the value vector and optimal stationary strategies for structured stochastic games for discounted and undiscounded zero-sum games and quadratic Multi-objective programming problem are formulated as linear complementary problems. For details see \cite{articlee18}, \cite{mondal2016discounted}, \cite{neogy2005linear} and \cite{neogy2008mixture}. The complementarity problems are considered with respect to principal pivot transforms and pivotal method to its solution point of view. For details see \cite{articlee8}, \cite{articlee10}, \cite{das1} and \cite{neogy2012generalized}. 

There so many methods are developed to solve a nonlinear complementarity problem. see \cite{pang1986}, \cite{pang1993},\cite{karamardian1969}, \cite{watson1979}.  Eaves and Saigal \cite{eaves1972homotopies} formed an important class of globally convergent methods for solving systems of non-linear equations. Such methods have been used to constructively prove the existence of solutions to many economic and engineering problems. The fundamental idea of a homotopy continuation method is to solve a problem by tracing a certain continuous path that leads to a solution of the problem. Thus, defining a homotopy mapping that yields a finite continuation path plays an essential role in a homotopy continuation method.

\vsp
The paper is organized as follows. Section 2 presents some basic notations and results. In section 3, we propose a new homotopy function to find the solution of nonlinear complementarity problem. We construct a smooth and bounded homotopy path  to find the solution of the nonlinear complementarity problem as the homotopy parameter $\lambda$ tends to $0$. To find the solution of homotopy function we use modified homotopy continuation method to increase the order of convergency of the algorithm.  We also find the sign of the positive tangent direction of the homotopy path. Finally, in section 4, we  numerically solve the oligopoly problem which is formulated by nonlinear complementarity problem using the introduced homotopy function. 

\section{Preliminaries}
Consider a function $f: R^n \rightarrow R^n$ , and a vector 
$z\in R^n$ such that $f= \left[\begin{array}{c}
f_1\\
f_2\\
\vdots\\
f_n\\
\end{array}\right]$ and $z= \left[\begin{array}{c}
z_1\\
z_2\\
\vdots\\
z_n\\
\end{array}\right].$ 
The complementarity problem is to find a vector $z\in R^n$ such that
\begin{equation}\label{cp}
z^Tf(z)=0, \ \ \  f(z)\geq 0 , \ \ \   z\geq 0.
\end{equation}
When the function $f$ is a nonlinear function, then it is called nonlinear complementarity problem.
\vsp
The basic idea of homotopy method is to construct a homotopy continuation path 
from the auxiliary mapping g to the object mapping f. 
Suppose the given problem is to find a root of the non-linear equation f(x) = 0
and suppose g(x) = 0 is an auxiliary equation with $g(x_0)=0$. Then the
homotopy function $H:R^{n+1} \to R^n$ can be
 defined as $H(x, \lambda) = $
 $ (1-\lambda)f(x) + \lambda g(x),$
 $ 0 \leq \lambda \leq 1.$
Then we consider the homotopy equation $H(x, \lambda) = 0,$ where $(x_0,1)$ is a known solution of the homotopy equation. Our aim is to find the solution of the equation $f(x)=0$ from the known solution of $g(x) = 0$ by solving the homotopy equation $H(x, \lambda) = 0$ varrying the values of $\lambda$  from $1$ to $0$. Kojima et al showed that under some conditions nonlinear complementarity problem can be solvable by homotopy continuation method. For details see \cite{kojima1991homotopy}, \cite{kojima1993general}, \cite{kojima1994global}, \cite{tseng1997infeasible}.
\vsp

Now we state some results which will be required in the next section.
\begin{lemma}
	(Generalizations of Sard's Theorem\cite{Chow}) \ Let $U \subset R^n$ be an open set and $f :R^n \to R^p$ be smooth. We say $y \in R^p$ is a regular value for $f$ if $\text{Range} Df(x) = R^p $ $\forall x \in f^{-1}(y),$ where $Df(x)$ denotes the $n \times p$ matrix of partial derivatives of $f(x).$
\end{lemma}
\begin{lemma}\label{par}
	(Parameterized Sard Theorem \cite{Wang}) \ Let $V \subset R^n, U \subset R^m$ be open sets, and let $\phi:V\times U \to R^k$ be a $C^\alpha$ mapping, where $\alpha >\text{max}\{0,m-k\}.$ If $0\in R^k$ is a regular value of $\phi,$ then for almost all $a \in V, 0$ is a regular value of $\phi _ a=\phi(a,.).$    
\end{lemma}
\begin{lemma}\label{inv}
	(The inverse image theorem \cite{Wang}) \ Let $\phi : U \subset R^n \to R^p$ be $C^\alpha$ mapping, where $\alpha >\text{max}\{0,n-p\}.$ Then $\phi^{-1}(0)$ consists of some $(n-p)$-dimensional $C^\alpha$ manifolds. 
\end{lemma}
\begin{lemma}\label{cl}
	(Classification theorem of one-dimensional smooth manifold \cite{N}) \ One-dimensional smooth manifold is diffeomorphic to a unit circle or a unit interval.
\end{lemma}

\begin{lemma}\cite{cordero2012increasing}\label{coorder}
 Consider the function $f:R^n\to R^n$ and the iterative method
 $y^k=x^k-f'(x^k)^{-1}f(x^k),$ \ 
 $z^k=x^k-2(f'(y^k)+f'(x^k))^{-1}f(x^k), \ w^k=z^k-f'(y^k)^{-1}f(z^k)$   has $5$th order convergence.
\end{lemma}
   \section{Main Results} 
     Now we solve nonlinear complementarity problem by homotopy method. Consider two positive numbers $m,n,$ such that $m$ is very large positive number and $l$ is positive number, $l<<m.$ 
 First we define\\ $\mathcal{R}_{(m)}=\{(z,y,w_1,w_2,v_1,v_2)\in R_{++}^n \times R_{++}^n \times R_{++}^n \times R_{++}^n \times R_{++} \times R_{++}:m-(\overset{n}{\underset{i=1}{\sum}}(z+w_1)_i+v_2)>l, m-(\overset{n}{\underset{i=1}{\sum}}(y+w_2)_i+v_1)>l\}, \ $\\
$ \mathcal{\bar{R}}_{(m)}=\{(z,y,w_1,w_2,v_1,v_2)\in R_{+}^n \times R_{+}^n \times R_{+}^n \times R_{+}^n \times R_{+} \times R_{+}:m-(\overset{n}{\underset{i=1}{\sum}}(z+w_1)_i+v_2)\geq l, m-(\overset{n}{\underset{i=1}{\sum}}(y+w_2)_i+v_1)\geq l\}.$\\ Here $m$ is a predefined large number. We choose the initial point\\ $x^{(0)}=(z^{(0)},y^{(0)},{w_1}^{(0)},{w_2}^{(0)},{v_1}^{(0)},{v_2}^{(0)})\in \mathcal{R}_{(m)}$ such that\\ 
 $A^0 B^0-B^0{v_2}^{(0)}-{A^{(0)}{v_1}^{(0)}}=0, $\\
$l(B^{(0)}{v_2}^{(0)}-A^{(0)}{v_1}^{(0)})+lB^{(0)}(l-A^{(0)})+A^{(0)}{v_1}^{(0)}(A^{(0)}-{v_2}^{(0)}) \neq 0,\\ l(A^{(0)}{v_1}^{(0)}-B^{(0)}{v_2}^{(0)})+lA^{(0)}(l-B^{(0)})+B^{(0)}{v_2}^{(0)}(B^{(0)}-{v_1}^{(0)}) \neq 0,\\ l(B^{(0)}-l)\neq(A^{(0)}-{v_2}^{(0)}){v_1}^{(0)},  l(A^{(0)}-l)\neq(B^{(0)}-{v_1}^{(0)}){v_2}^{(0)},$\\ where $A=(m-(\overset{n}{\underset{i=1}{\sum}}(z+w_1)_i)), A^{(0)}=(m-(\overset{n}{\underset{i=1}{\sum}}(z^{(0)}+w_1^{(0)})_i)), \\
B=(m-(\overset{n}{\underset{i=1}{\sum}}(y+w_2)_i)),$ $B^{(0)}=m-(\overset{n}{\underset{i=1}{\sum}}(y^{(0)}+w_2^{(0)})_i). $ \\Now define the feasible region $\mathcal{F}_{(m)}=\{(z,y,w_1,w_2,v_1,v_2)\in {\mathcal{R}}_{(m)}: {v_1} \neq \frac{\lambda(A^{(0)}-{v_2}^{(0)}){v_1}^{(0)}}{l}; \\ {v_2} \neq \frac{\lambda(B^{(0)}-{v_1}^{(0)}){v_2}^{(0)}}{l} \ \forall  \ \lambda \in (0,1)\},$\\ $\mathcal{\bar{F}}_{(m)}=\{(z,y,w_1,w_2,v_1,v_2) \in \mathcal{\bar{R}}_{(m)}:  {v_1} \neq \frac{\lambda(A^{(0)}-{v_2}^{(0)}){v_1}^{(0)}}{l};  {v_2} \neq \frac{\lambda(B^{(0)}-{v_1}^{(0)}){v_2}^{(0)}}{l} \ \forall \ \lambda \in (0,1)\}.$ \\
$\partial \mathcal{F}_{(m)}=\{(z,y,w_1,w_2,v_1,v_2)\in \partial\mathcal{R}_{(m)}: {v_1} \neq \frac{\lambda(A^{(0)}-{v_2}^{(0)}){v_1}^{(0)}}{l};  {v_2} \neq \frac{\lambda(B^{(0)}-{v_1}^{(0)}){v_2}^{(0)}}{l} \ \forall \ \lambda \in (0,1)\},$ where $\partial\mathcal{R}_{(m)}$ is the boundary of $ \mathcal{\bar{R}}_{(m)}.$
 
 Now we construct the homotopy function 
 \begin{equation}\label{homf}
 H(x,x^{(0)},\lambda)=\left[\begin{array}{c} 
 (1-\lambda)(y-w_1+v_1e+J_f^t(z-w_2+v_2e))+\lambda(z-z^{(0)}) \\
 W_1z-\lambda W_1^{(0)}z^{(0)}\\
 W_2y-\lambda W_2^{(0)}y^{(0)}\\
 y-(1-\lambda)f(z)-\lambda(y^{(0)})\\
 (m-\overset{n}{\underset{i=1}{\sum}}(z+w_1)_i-v_2)v_1- \lambda((m-\overset{n}{\underset{i=1}{\sum}}(z^{(0)}+w_1^{(0)})_i-v_2^{(0)})v_1^{(0)})\\
 (m-\overset{n}{\underset{i=1}{\sum}}(y+w_2)_i-v_1)v_2- \lambda((m-\overset{n}{\underset{i=1}{\sum}}(y^{(0)}+w_2^{(0)})_i-v_1^{(0)})v_2^{(0)})\\
 \end{array}\right]=0
 \end{equation}
 where $e=[1,1,\cdots,1]^t,\ Z=\text{diag}(z);\ W_1=\text{diag}(w_1); \ W_2=\text{diag}(w_2); \ W_1^{(0)}=\text{diag}(w_1^{(0)}); \ W_2^{(0)}=\text{diag}(w_2^{(0)}); \ x=(z,y,w_1,w_2,v_1,v_2)\in \mathcal{\bar{F}}_{(m)}; \ x^{(0)}=(z^{(0)},y^{(0)},{w_1}^{(0)},{w_2}^{(0)},{v_1}^{(0)},{v_2}^{(0)})\in \mathcal{F}_{(m)}; \ \lambda \in (0,1]$ and $J_f$ is the jacobian of $f(z).$ 

\begin{theorem}\label{reg}
	For almost all initial points $x^{(0)}\in \mathcal{F}_{(m)},$ $0$ is a regular value of the homotopy function $H:R^{4n+2} \times (0,1] \to R^{4n+2}$ and the zero point set $H_{x^{(0)}}^{-1}(0)=\{(x,\lambda)\in \mathcal{F}_{(m)} \times (0,1]:H_{x^{(0)}}(x,\lambda)=0\}$ contains a smooth curve $\Gamma_x^{(0)}$ starting from $(x^{(0)},1).$  
\end{theorem}
\begin{proof}
	The jacobian matrix of the above homotopy function $H(x,x^{(0)},\lambda)$ is denoted by $DH(x,x^{(0)},\lambda))$ and we have $DH(x,x^{(0)},\lambda))=$$\left[\begin{array}{ccc} 
	\frac{\partial{H(x,x^{(0)},\lambda)}}{\partial{x}} & 	\frac{\partial{H(x,x^{(0)},\lambda)}}{\partial{x^{(0)}}} & \frac{\partial{H(x,x^{(0)},\lambda)}}{\partial{\lambda}}\\ 
	\end{array}\right].$ For all $x^{(0)} \in \mathcal{F}_1$ and $\lambda \in (0,1],$ we have $\frac{\partial{H(x,x^{(0)},\lambda)}}{\partial{x^{(0)}}}=$$\begin{bmatrix}
	K_1 & K_2\\
	K_3 & K_4\\
	\end{bmatrix},$\\ where $K_1=\begin{bmatrix} -\lambda I & 0 & 0 & 0\\
    -\lambda W_1^{(0)} & 0 & -\lambda Z^{(0)} & 0\\
     0 & -\lambda W_2^{(0)} & 0 & -\lambda Y^{(0)}\\
      0 & -\lambda I & 0 & 0\\
      \end{bmatrix},$ $K_2=\begin{bmatrix}
      0 & 0\\
      0 & 0\\
      0 & 0\\
      0 & 0\\
       \end{bmatrix},$ 
       
       $K_3=\begin{bmatrix}
       \lambda v_1^{(0)} e^t & 0 & \lambda v_1^{(0)} e^t & 0 \\
       0 & \lambda v_2^{(0)} e^t  & 0 & \lambda v_2^{(0)} e^t\\
       \end{bmatrix},$ 
       
       $K_4= \begin{bmatrix}
       -\lambda(m-\overset{n}{\underset{i=1}{\sum}}(z^{(0)}+w_1^{(0)})_i-v_2^{(0)}) & \lambda v_1^{(0)}\\
       \lambda v_2^{(0)} & -\lambda(m-\overset{n}{\underset{i=1}{\sum}}(y^{(0)}+w_2^{(0)})_i-v_1^{(0)})\\
       \end{bmatrix} ,$ \\
     $Y^{(0)}=\text{diag}(y^{(0)}), Z^{(0)}=\text{diag}(z^{(0)})$, $W_1^{(0)}=\text{diag}(w_1^{(0)})$, $W_2^{(0)}=\text{diag}(w_2^{(0)})$.\\
	$\text{det}(\frac{\partial{H}}{ \partial{x^{(0)}}})=\lambda^{4n+2}((m-\overset{n}{\underset{i=1}{\sum}}(z^{(0)}+w_1^{(0)})_i-v_2^{(0)})(m-\overset{n}{\underset{i=1}{\sum}}(y^{(0)}+w_2^{(0)})_i-v_1^{(0)})-v_1^{(0)}v_2^{(0)})\prod_{i=1}^{n} z_i^{(0)}y_i^{(0)}\neq 0$ for $\lambda \in (0,1].$ \\ 
	Thus $DH(x,x^{(0)},\lambda))$ is of full row rank. Therefore, $0$ is a regular value of $H(x,x^{(0)},\lambda)).$ By Lemmas \ref{par} and \ref{inv}, for almost all  $x^{(0)} \in \mathcal{F}_{(m)},$ $0$ is a regular value of $H_{x^{(0)}}(x,\lambda)$ and $H_{x^{(0)}}^{-1}(0)$ consists of some smooth curves and $H_{x^{(0)}}(x^{(0)},1)=0.$ Hence there must be a smooth curve  $\Gamma_x^{(0)}$ starting from  $(x^{(0)},1).$
\end{proof}

\begin{theorem}\label{bdd}
	Let $\mathcal{F}_{(m)}$ be a nonempty set. For a given $x^{(0)} \in \mathcal{F}_{(m)},$ if  $0$ is a regular value of $H(x,x^{(0)},\lambda),$ then $\Gamma_x^{(0)}$ is a bounded curve in $\mathcal{\bar{F}}_{(m)} \times (0,1].$ 
\end{theorem}
\begin{proof}
		We have that $0$ is a regular value of $H(x,x^{(0)},\lambda)$  by theorem \ref{reg} and $\mathcal{F}_{(m)}$ be a nonempty set. It is clear that the set $\mathcal{F}_{(m)}$ and $(0,1]$ is bounded. So there exists a sequence of points $\{z^k,y^k,w_1^k,w_2^k,v_1^k,v_2^k,\lambda^k\} \subset \Gamma_x^{(0)} \times (0,1],$ such that $\lim\limits_{k \to \infty}z^k=\bar{z}, \lim\limits_{k \to \infty}y^k=\bar{y}, \lim\limits_{k \to \infty}w_1^k=\bar{w_1}, \lim\limits_{k \to \infty}w_2^k=\bar{w_2}, \lim\limits_{k \to \infty}v_1^k=\bar{v_1}, \lim\limits_{k \to \infty}v_2^k=\bar{v_2}, \lim\limits_{k \to \infty}\lambda^k=\bar{\lambda}. $ Hence $\Gamma_x^{(0)}$ is a bounded curve in $\mathcal{\bar{F}}_{(m)} \times (0,1].$

\end{proof}

Now we show the convergence of the homotopy function \ref{homf}. 
\begin{theorem}
	For $x^{(0)}=(z^{(0)},y^{(0)},w_1^{(0)},w_2^{(0)},v_1^{(0)},v_2^{(0)})\in  \mathcal{R}_{(m)}$ such that\\
	 $A^0 B^0-B^0{v_2}^{(0)}-{A^{(0)}{v_1}^{(0)}}=0,$\\
	$l(B^{(0)}{v_2}^{(0)}-A^{(0)}{v_1}^{(0)})+lB^{(0)}(l-A^{(0)})+A^{(0)}{v_1}^{(0)}(A^{(0)}-{v_2}^{(0)}) \neq 0,\\ l(A^{(0)}{v_1}^{(0)}-B^{(0)}{v_2}^{(0)})+lA^{(0)}(l-B^{(0)})+B^{(0)}{v_2}^{(0)}(B^{(0)}-{v_1}^{(0)}) \neq 0,\\ l(B^{(0)}-l)\neq(A^{(0)}-{v_2}^{(0)}){v_1}^{(0)}, \\ l(A^{(0)}-l)\neq(B^{(0)}-{v_1}^{(0)}){v_2}^{(0)},$ \\ the homotopy equation finds a bounded smooth curve $\Gamma_x^{(0)} \subset \mathcal{{F}}_{(m)} \times (0,1]$ which starts from $(x^{(0)},1)$ and approaches the hyperplane at $\lambda \to 0.$ As $\lambda \to 0,$ the limit set $\mathcal{L} \times \{0\} \subset \mathcal{\bar{F}}_{(m)} \times \{0\}$ of $\Gamma_x^{(0)}$ is nonempty and every point in $\mathcal{L}$ is a solution of the following system of equations:
	\begin{equation}\label{sys}
	\begin{aligned}
		(y-w_1+v_1e+J_f^t(z-w_2+v_2e))=0 \\
	W_1z=0\\
	W_2y=0\\
	y-f(z)=0\\
	(m-\overset{n}{\underset{i=1}{\sum}}(z+w_1)_i-v_2)v_1=0\\
	(m-\overset{n}{\underset{i=1}{\sum}}(y+w_2)_i-v_1)v_2=0\\
	\end{aligned}
	\end{equation}
\end{theorem}
\begin{proof}

	Note that $\Gamma_x^{(0)}$ is diffeomorphic to a unit circle or a unit interval $(0,1]$ in view of Lemma \ref{cl}. As $\frac{\partial{H(x,x^{(0)},1)}}{\partial{x^{(0)}}}$ is nonsingular, $\Gamma_x^{(0)}$ is diffeomorphic to a unit interval $(0,1].$ Again $\Gamma_x^{(0)}$ is a bounded smooth curve by the Theorem \ref{bdd}. Let $(\bar{x},\bar{\lambda})$ be a limit point of $\Gamma_x^{(0)}.$ We consider four cases:
	\begin{enumerate}
		\item[(i)] $(\bar{x},\bar{\lambda})\in \mathcal{{F}}_{(m)} \times \{1\}.$
		\item[(ii)] $(\bar{x},\bar{\lambda})\in \partial{\mathcal{{F}}_{(m)}} \times \{1\}.$
		\item[(iii)] $(\bar{x},\bar{\lambda})\in \partial{\mathcal{{F}}_{(m)}} \times (0,1).$
		\item[(iv)] $(\bar{x},\bar{\lambda})\in \mathcal{\bar{F}}_{(m)} \times \{0\}.$
	\end{enumerate}
	 Suppose for case (i) the homotopy function \ref{homf} has solution $(\bar{x},1)$, other than the initial solution  $x^{(0)}$. As $\lambda \to 1$, $\bar{y}=y^{(0)}, \bar{z_1}=z_1 ^{0}, \bar{z_2}=z_2 ^{0},\bar{v_1} \neq 0, \bar{v_2} \neq 0$. So for $\lambda \to 1$, $(A-v_2) \to (A^{(0)}-\bar{v_2}), (B-v_1)  \to (B^{(0)}-\bar{v_1})$. Hence from homotopy function \ref{homf}
	 \begin{eqnarray}\label{t1}
	 (A^{(0)}-\bar{v_2})\bar{v_1}=(A^0-{v_2}^{(0)}){v_1}^{(0)}\\ \label{t2}
	 (B^{(0)}-\bar{v_1})\bar{v_2}=(B^0-{v_1}^{(0)}){v_2}^{(0)} \label{t3}
	 \end{eqnarray}
	    From \ref{t3}, $\bar{v_2}=\frac{B^0-{v_1}^{(0)}}{B^{(0)}-\bar{v_1}}{v_2}^{(0)}.$ From equation \ref{t1} $(A^{(0)}-\frac{B^0-{v_1}^{(0)}}{B^{(0)}-\bar{v_1}}{v_2}^{(0)})\bar{v_1}=(A^0-{v_2}^{(0)}){v_1}^{(0)}.$ This implies that $\bar{v_1}=\frac{-(B^0{v_2}^{(0)}-A^0{v_1}^{(0)}-A^0 B^0) \pm \sqrt{(A^0 B^0-B^0{v_2}^{(0)}-A^0{v_1}^{(0)})^2}}{2A^0}$.\\
	    $\implies \bar{v_1}={v_1}^{(0)} \text{or} $ $\frac{A^0 B^0-B^0{v_2}^{(0)} }{A^{(0)}}$. As $\bar{v_1}=\frac{A^0 B^0-B^0{v_2}^{(0)} }{A^{(0)}}={v_1}^{(0)} $ from the condition of choosing the initial point $x^{(0)}, $ the equation $H_{x^{(0)}}(x,1)=0$ has only one solution $x^{(0)}\in  \mathcal{{R}}_{(m)}.$ Hence the case $(i)$ is impossible.\\
	    In case $(ii)$ the homotopy equation \ref{homf} implies that $ \bar{y}=y^{(0)}, \bar{z_1}=z_1 ^{0}, \bar{z_2}=z_2 ^{0}, \bar{v_1} \neq 0, \bar{v_2} \neq 0.$ So $(A-v_2) \to (A^{(0)}-\bar{v_2})$ and $(B-v_1) \to (B^{(0)}-\bar{v_1})$ as $\lambda \to 1.$ From last two components of homotopy equation \ref{homf} we have \begin{equation}\label{ineq}
	(A^{(0)}-\bar{v_2})\bar{v_1}=(A^{(0)}-{v_2}^{(0)}){v_1}^{(0)}, 
	 (B^{(0)}-\bar{v_1})\bar{v_2}=(B^{(0)}-{v_1}^{(0)}){v_2}^{(0)}.
	\end{equation} Three cases may arise. \\
	\textbf{Case 1:} Let $A^{(0)}-\bar{v_2}=l.$ From equation \ref{ineq} \\ $\bar{v_1}=\frac{(A^{(0)}-{v_2}^{(0)}){v_1}^{(0)}}{l},\  (B^{(0)}-\frac{(A^{(0)}-{v_2}^{(0)}){v_1}^{(0)}}{l})\bar{v_2}=(B^{(0)}-{v_1}^{(0)}){v_2}^{(0)} \\ \implies \bar{v_2}=$ $\frac{l(B^{(0)}-{v_1}^{(0)}){v_2}^{(0)}}{lB^{(0)}-A^{(0)}{v_1}^{(0)}+{v_1}^{(0)}{v_2}^{(0)}}=A^{(0)}-l \\ \implies l(B^{(0)}{v_2}^{(0)}-A^{(0)}{v_1}^{(0)})+lB^{(0)}(l-A^{(0)})+A^{(0)}{v_1}^{(0)}(A^{(0)}-{v_2}^{(0)}) = 0,$ contradicts the choosing of initial point.\\	
	\textbf{Case 2:}  Let $B^{(0)}-\bar{v_1}=l.$ From equation \ref{ineq} \\ $\bar{v_2}=\frac{(B^{(0)}-{v_1}^{(0)}){v_2}^{(0)}}{l}, \  (A^{(0)}-\frac{(B^{(0)}-{v_1}^{(0)}){v_2}^{(0)}}{l})\bar{v_1}=(A^{(0)}-{v_2}^{(0)}){v_1}^{(0)} \\ \implies \bar{v_1}=$ $\frac{l(A^{(0)}-{v_2}^{(0)}){v_1}^{(0)}}{lA^{(0)}-B^{(0)}{v_2}^{(0)}+{v_1}^{(0)}{v_2}^{(0)}}=B^{(0)}-l \\ \implies l(A^{(0)}{v_1}^{(0)}-B^{(0)}{v_2}^{(0)})+lA^{(0)}(l-B^{(0)})+B^{(0)}{v_2}^{(0)}(B^{(0)}-{v_1}^{(0)}) = 0,$ contradicts the choosing of initial point.\\
		\textbf{Case 3:}  Let $B^{(0)}-\bar{v_1}=l, \  A^{(0)}-\bar{v_2}=l. $ From equation \ref{ineq} \\ we have $l\bar{v_1}=(A^{(0)}-{v_2}^{(0)}){v_1}^{(0)},l\bar{v_2}=(B^{(0)}-{v_1}^{(0)}){v_2}^{(0)}.\\ \implies l(B^{(0)}-l)=(A^{(0)}-{v_2}^{(0)}){v_1}^{(0)}, l(A^{(0)}-l)=(B^{(0)}-{v_1}^{(0)}){v_2}^{(0)}, $ contradicts the choosing of initial point.
	
	In case $(iii)$ from homotopy equation \ref{homf} we have $\bar{z}>0,\bar{y}>0,\bar{w_1}>0,\bar{w_2}>0.$ Three cases may arise.\\
	\textbf{Case1:} Let $A-v_2 \to \bar{A}-\bar{v_2}=l,$ where $\bar{A}=(m-\overset{n}{\underset{i=1}{\sum}}(\bar{z}+\bar{w_1})_i.$ Then from equation \ref{ineq} we have $\bar{v_1}=\frac{\bar{\lambda}(A^{(0)}-{v_2}^{(0)}){v_1}^{(0)}}{l},$ which contradicts that $\bar{v_1}\in  \partial{\mathcal{{F}}_{(m)}}.$\\
	\textbf{Case2:} Let $B-v_1 \to \bar{B}-\bar{v_1}=l,$ where $\bar{B}=(m-\overset{n}{\underset{i=1}{\sum}}(\bar{y}+\bar{w_2})_i.$  Then from equation \ref{ineq} we have $\bar{v_2}=\frac{\bar{\lambda}(B^{(0)}-{v_1}^{(0)}){v_2}^{(0)}}{l},$ which contradicts that $\bar{v_2}\in  \partial{\mathcal{{F}}_{(m)}}.$\\
	\textbf{Case3:} Let $A-v_2 \to \bar{A}-\bar{v_2}=l,B-v_1 \to \bar{B}-\bar{v_1}=l.$ Then from equation \ref{ineq} we have $\bar{v_1}=\frac{\bar{\lambda}(A^{(0)}-{v_2}^{(0)}){v_1}^{(0)}}{l}$ and $\bar{v_2}=\frac{\bar{\lambda}(B^{(0)}-{v_1}^{(0)}){v_2}^{(0)}}{l},$ which contradicts that $\bar{v_1}\in  \partial{\mathcal{{F}}_{(m)}}$ and $\bar{v_2}\in  \partial{\mathcal{{F}}_{(m)}}.$\\
	  Therefore $(iv)$ is the only possible case. Hence $\bar{x}=(\bar{z},\bar{y},\bar{w_1},\bar{w_2},\bar{v_1},\bar{v_2})$ is a solution of the system of equations \ref{sys}.
	 \end{proof}
\begin{remk}\label{mmf}
	From the homotopy function \ref{homf}  as $\lambda \to 0$ we get $\bar{y}-\bar{w}_1+\bar{J}_f^t(\bar{z}-\bar{w}_2)=0,$ $\bar{y}=f(\bar{z})$ and
$\bar{w}_{1i}\bar{z}_i=0,$ $\bar{w}_{2i}\bar{y}_i=0 \ \forall i\in\{1,2, \cdots n\}$, where $\bar{J}_f$ is the jacobian of $f(z)$ at the point $\bar{z}.$ Now $\bar{w}_1$ and $\bar{w}_2$ can be decomposed as $\bar{w}_1= \bar{y}-\Delta\bar{y} \geq 0$  and $\bar{w}_2= \bar{z}-\Delta\bar{z} \geq 0.$ Now it is clear that $\bar{y}_i \bar{z}_i=\Delta\bar{y}_i \bar{z}_i=\Delta\bar{z}_i\bar{y}_i \ \forall i \in \{1,2, \cdots n\}$ and $\bar{J}_f^t\Delta \bar{z}+\Delta \bar{y}=0.$ This implies that $(\bar{Z}\bar{J}_f^t+\bar{Y})\Delta \bar{z}=0,$ where  $\bar{Y}=$diag$(\bar{y})$ and $\bar{Z}=$diag$(\bar{z}).$
\end{remk}

\begin{theorem}
	The component $\bar{z}$ of $(\bar{z},\bar{y},\bar{w}_1,\bar{w}_2,0) \in \mathcal{L}\times \{0\}$ gives the solution of the complementarity problem\ref{cp}  if and only if $\Delta\bar{z}_i \Delta\bar{y}_i=0 $ or $\bar{w}_{1i}+\bar{w}_{2i}>0  \ \forall i \in \{1,2,\cdots n\}.$  
\end{theorem} 
\begin{proof}
		Suppose $\bar{z} \geq 0$ and $\bar{y}=f(\bar{z}) \geq 0$ give the solution of the complementarity problem\ref{cp}. Then $\bar{z}_i\bar{y}_i=0$ \ $ \forall i \in \{1,2,\cdots n\}.$ This implies that $\bar{z}_i=0$ or $\bar{y}_i=0$ \ $ \forall i \in \{1,2,\cdots n\}.$ Now we consider the following cases.\\
	\NI	Case 1:	For atleast one $i \in  \{1,2,\cdots n\},$ let $\bar{z}_{i}>0, \bar{y}_{i}=0.$  In view of Remark \ref{mmf}, this implies that $\Delta\bar{y}_i=0 \implies \Delta\bar{z}_i \Delta\bar{y}_i=0.$  \\
	\NI	Case 2:	For atleast one $i \in  \{1,2,\cdots n\},$ let $\bar{y}_{i}>0, \bar{z}_{i}=0.$ In view of \ref{mmf}, this implies that $\Delta\bar{z}_i=0 \implies \Delta\bar{z}_i \Delta\bar{y}_i=0.$  \\
	\NI	Case 3:	For atleast one $i \in  \{1,2,\cdots n\},$ let $\bar{y}_{i}=0, \bar{z}_{i}=0.$  This implies that either $\Delta\bar{y}_i\Delta\bar{z}_i=0$ or $\bar{w}_{1i}+\bar{w}_{2i}>0.$ 
	
	Conversely, let consider $\Delta\bar{z}_i \Delta\bar{y}_i=0 $ or $\bar{w}_{1i}+\bar{w}_{2i}>0 \ \forall i \in \{1,2,\cdots n\}.$ Let $\forall i \in \{1,2,\cdots n\}, \ \Delta\bar{z}_i \Delta\bar{y}_i=0 $ implies either $\Delta\bar{z}_i=0$ or $\Delta\bar{y}_i=0.$ This implies that $\bar{y}_i \bar{z}_i=0 \  \forall i \in \{1,2,\cdots n\}.$ Therefore $\bar{y}$ and $\bar{z}$ give the solution of given complementarity problem \ref{cp}.
	\NI	Let consider $\bar{w}_{1}+\bar{w}_{2}>0.$ Then three cases will arise.
	
	\NI	Case 1:	Let $\bar{w}_{1i}>0, \bar{w}_{2i}=0$ for atleast one $i \in  \{1,2,\cdots n\}.$ This implies that $\bar{z}_i=0$ and $\bar{y}_i \geq 0.$ \\
	Case 2: Let $\bar{w}_{1i}=0, \bar{w}_{2i}>0$ for atleast one $i \in  \{1,2,\cdots n\}.$ This implies that $\bar{z}_i \geq 0$ and $\bar{y}_i=0.$ \\
	Case 3:  Let $\bar{w}_{1i}>0, \bar{w}_{2i}>0$ for atleast one $i \in  \{1,2,\cdots n\}.$ This implies that $\bar{z}_i=0$ and $\bar{y}_i=0.$ \\
	Considering the above three cases $\bar{z}$ and $\bar{y}$ solves the compplementarity problem \ref{cp}.\\
  	\end{proof}
\begin{theorem}
	If the nonlinear function $f(z)$ is a $P_0$ function, then the component $\bar{z}$ of $(\bar{z},\bar{y} ,\bar{w}_1,\bar{w}_2,0) \in \mathcal{L}\times \{0\}$ gives the solution of the nonlinear complementarity problem \ref{cp}.
\end{theorem}
\begin{proof}
	Let $f(z)$ be a $P_0$ function. Then the jacobian matrix of the nonlinear function at a point $z$,  $\bar{J}_f$ is a $P_0$ matrix. Assume that the component $\bar{z}$ of $(\bar{z},\bar{y},\bar{w}_1,\bar{w}_2,0) \in \mathcal{L}\times \{0\}$ does not give the solution of the nonlinear complementarity problem  \ref{cp}. Hence $\Delta\bar{z}_i \Delta\bar{y}_i\neq0 $ and $\bar{w}_{1i}+\bar{w}_{2i}=0$ for atleast one $i$. Then $\Delta\bar{z}_i \neq 0, \Delta\bar{y}_i\neq0 , \bar{w}_{1i}=0, \bar{w}_{2i}=0.$  Now $\bar{w}_{1i}=\bar{y}_i-\Delta\bar{y}_i=0$ and $\Delta\bar{z}_i \Delta\bar{y}_i\neq0$ $\implies \bar{y}_i=\Delta\bar{y}_i>0$. In similar way $\bar{w}_{2i}=\bar{z}_i-\Delta\bar{z}_i=0$ and $\Delta\bar{z}_i \Delta\bar{y}_i\neq0$ $\implies \bar{z}_i=\Delta\bar{z}_i>0$. As $(\bar{z},\bar{y},\bar{w}_1,\bar{w}_2,0) \in  \mathcal{\bar{F}}_{(m)}\times \{0\}$, $v_1=0, v_2=0$.  From Equation \ref{sys}, $\Delta\bar{y}_i + ({\bar{J}_f}^t\Delta\bar{z})_i=0.$  This implies that $ ({\bar{J}_f}^t\Delta\bar{z})_i<0$ and also $(\bar{z})_i ({\bar{J}_f}^t\Delta\bar{z})_i<0.$ This contradicts that $\bar{J}_f$ is a $P_0$-matrix. Therefore the component $\bar{z}$ of $(\bar{z},\bar{y},\bar{w}_1,\bar{w}_2,0) \in \mathcal{L}\times \{0\}$ gives the solution of the nonlinear complementarity problem \ref{cp}.
\end{proof}

\begin{theorem}\label{22222}
	Suppose the matrix $(\bar{Z}\bar{J}_f^t+\bar{Y})$  is nonsingular, where $\bar{Y}=$diag$(\bar{y})$ and $\bar{Z}=$diag$(\bar{z}).$ Then $\bar{z}$ solves the complementarity problem \ref{cp}.		 	
\end{theorem}
\begin{proof}
		Let $(\bar{Z}\bar{J}_f^t+\bar{Y})$ is nonsingular matrix. Now from remark \ref{mmf} it is clear that $\Delta z=0.$  This implies that $\bar{z}$ solves the complementarity problem \ref{cp}.
\end{proof}
\begin{theorem}
    If the jacobian matrix $\bar{J}_f$ has nonsingular principal minors, then the component $\bar{z}$ of $(\bar{z}, \bar{y},\bar{w}_1,\bar{w}_2,0) \in \mathcal{L}\times \{0\}$ solves the nonlinear complementarity problem \ref{cp}.
	\end{theorem}
	\begin{proof}
		Consider that the jacobian matrix $\bar{J}_f$ has nonsingular principal minors. Then the transpose of the jacobian matrix $\bar{J}_f ^t$ has nonsingular principal minors. By theorem \ref{22222}, if the matrix $(\bar{Z}\bar{J}_f^t+\bar{Y})$ is nonsingular, then $\bar{z}$ solves the nonlinear complementarity problem, where $\bar{Y}=\text{diag}(\bar{y}),$ $\bar{Z}=\text{diag}(\bar{z}).$  Let $\tilde{\mathcal{A'}}=\left[\begin{array}{cc} 
			\bar{Y} & \bar{Z}\\
			-\bar{J}_f^t & I\\
		\end{array}\right]$.  Then $\det(\tilde{\mathcal{A'}})= \det(\bar{Y}+\bar{Z}\bar{J}_f^t)$.  Assume that the component $\bar{z}$ of $(\bar{z},\bar{y},\bar{w}_1,\bar{w}_2,0) \in \mathcal{L}\times \{0\}$ is not the solution of the nonlinear complementarity problem. Then there exists atleast one $i$, such that $\bar{z}_i\bar{y}_i>0$. Without loss of generality  $\bar{y}$ and $\bar{z}$ can be represented as $\bar{y}=\left[\begin{array}{c} 
			\bar{y}_p \\
			\bar{y}_q\\
			\bar{o}_r\\
		\end{array}\right]$, $\bar{z}=\left[\begin{array}{c} 
			\bar{o}_p \\
			\bar{z}_q\\
			\bar{z}_r\\
		\end{array}\right]$, where $\bar{y}_p\in {R^p}_{++}, \ \bar{y}_q, \bar{x}_q \in {R^q}_{++}, \ \bar{z}_r\in {R^r}_{++}$, \ $\bar{o}_r\in R^r, \ \bar{o}_p \in R^p$ and $\bar{o}_r, \ \bar{o}_p $ are vectors with all zeros. Here $(\bar{y}_q)_i(\bar{z}_q)_i>0$ and $\bar{Y}=\text{diag}(\bar{y}), \bar{Z}=\text{diag}(\bar{z})$.  Now we can rewrite $\left[\begin{array}{cc} 
			\bar{Y} & \bar{Z}\\
			-\bar{J}_f^t & I\\
		\end{array}\right]=\left[\begin{array}{cccccc} 
			\bar{Y}_p & \bar{O}_q & \bar{O}_r & \bar{O}_p & \bar{O}_q & \bar{O}_r \\
			\bar{O}_p & \bar{Y}_q & \bar{O}_r & \bar{O}_p & \bar{Z}_q & \bar{O}_r\\
			\bar{O}_p & \bar{O}_q & \bar{O}_r & \bar{O}_p & \bar{O}_q & \bar{Z}_r\\
			M' & B' & C' & I_p & \bar{O}_q & \bar{O}_r\\
			D' & E' & F' & \bar{O}_p & I_q & \bar{O}_r\\
			G' & H' & K' & \bar{O}_p & \bar{O}_q & I_r\\
		\end{array}\right]$, where\\ $-\bar{J}_f^t=\left[\begin{array}{ccc} 
			M' & B' & C'\\
			D' & E' & F' \\
			G' & H' & K' \\
		\end{array}\right]$, \ $\bar{Y}=\left[\begin{array}{ccc} 
			\bar{Y}_p & \bar{O}_q & \bar{O}_r \\
			\bar{O}_p & \bar{Y}_q & \bar{O}_r \\
			\bar{O}_p & \bar{O}_q & \bar{O}_r \\
		\end{array}\right]$, \ $\bar{Z}=\left[\begin{array}{ccc} 
			\bar{O}_p & \bar{O}_q & \bar{O}_r \\
			\bar{O}_p & \bar{Z}_q & \bar{O}_r\\
			\bar{O}_p & \bar{O}_q & \bar{Z}_r\\
		\end{array}\right]$, \ $\bar{Z}_q=\text{diag}(\bar{z}_q)$, \ $\bar{Z}_r=\text{diag}(\bar{z}_r)$, \ $\bar{Y}_q=\text{diag}(\bar{y}_q)$, $\bar{Y}_p=\text{diag}(\bar{y}_p)$, \ $\bar{O}_p=\text{diag}(\bar{o}_p)$, \ $\bar{O}_q=\text{diag}(\bar{o}_q)$ \ $\bar{O}_r=\text{diag}(\bar{o}_r)$, \ $M',D',G', I_p\in R^{p \times p}$, \ $B',E',H', I_q \in R^{q \times q}$, \ $C',F',K', I_r \in R^{r \times r}$ and $I_p,I_q,I_r$ are identity matrices.  By elementary row operations we can get
		\vsp
		$\tilde{\mathcal{B'}}=\left[\begin{array}{cccccc} 
			I & \bar{O}_q & \bar{O}_r & \bar{O}_p & \bar{O}_q & \bar{O}_r \\
			\bar{O}_p & I & \bar{O}_r & \bar{O}_p & \bar{Z}_q{\bar{Y}_q}^{-1} & \bar{O}_r\\
			\bar{O}_p & \bar{O}_q & \bar{O}_r & \bar{O}_p & \bar{O}_q & I\\
			M' & B' & C' & I & \bar{O}_q & \bar{O}_r\\
			D' & E' & F' & \bar{O}_p & I & \bar{O}_r\\
			G' & H' & K' & \bar{O}_p & \bar{O}_q & I\\
		\end{array}\right]$.
		\vsp
		By interchanging rows this matrix reduces to \\ $\tilde{\mathcal{C'}}=$ $\left[\begin{array}{cccccc} 
			I & \bar{O}_q & \bar{O}_r & \bar{O}_p & \bar{O}_q & \bar{O}_r \\
			\bar{O}_p & I & \bar{O}_r & \bar{O}_p & \bar{Z}_q{\bar{Y}_q}^{-1} & \bar{O}_r\\
			-G' & -H' & -K' & \bar{O}_p & \bar{O}_q & \bar{O}_r\\
			M' & B' & C' & I & \bar{O}_q & \bar{O}_r\\
			D' & E' & F' & \bar{O}_p & I & \bar{O}_r\\
			G' & H' & K' & \bar{O}_p & \bar{O}_q & I\\
		\end{array}\right]$.\\ Hence $\det(\tilde{\mathcal{A'}})= \det(\tilde{\mathcal{C'}})=(-1)^r\det(K')\neq 0$. Therefore by theorem \ref{22222} $\bar{z}$ solves the nonlinear complementarity problem. This contradicts our assumption. Hence the component $\bar{z}$ of $(\bar{z},\bar{y},\bar{w}_1,\bar{w}_2,0) \in \mathcal{L}\times \{0\}$ is the solution of the nonlinear complementarity problem \ref{cp}.
	\end{proof}

\begin{remk}
	Now we trace the homotopy path $\Gamma_x^{(0)} \subset \mathcal{{F}}_{(m)} \times (0,1]$ from the initial point $(x^{(0)},1)$ until $\lambda \to 0$ and find the solution of the given complementarity problem \ref{cp} under some assumptions. Let $s$ denote the arc length of $\Gamma_x^{(0)}$, we can parameterize the homotopy path $\Gamma_x^{(0)}$ with respect to $s$ in the following form\\
	\begin{equation}\label{ss}
	H_{x^{(0)}} (x(s),\lambda (s))=0, \ 
	x(0)=x^{(0)}, \   \lambda(0)=1.
	\end{equation} 
	Now differentiating \ref{ss} with respect to $s$ we obtain the following system of ordinary differential equations with given initial values\cite{fan}
	\begin{equation}
	H'_{x^{(0)}} (x(s),\lambda (s))\left[\begin{array}{c} 
	\frac{dx}{ds}\\
	\frac{d\lambda}{ds}\\
	\end{array}\right]=0, \
	\|( \frac{dx}{ds},\frac{d\lambda}{ds})\|=1, \ 
	x(0)=x^{(0)}, \   \lambda(0)=1, \ \frac{d\lambda}{ds}(0)<0, 
	\end{equation} 
	and the $x$-component of $(x(\bar{s}),\lambda (\bar{s}))$ gives the solution of the complementarity problem for $\lambda (\bar{s})=0.$
\end{remk}

Now we use the homotopy continuation method with some modifications to trace the homotopy path $\Gamma_x^{(0)}$ numerically. For details see ( cite our homotopy paper)
\subsection{Algorithm}\label{homosin}
\textbf{Step 0:} Set $i=i_s=0.$ [$i$ is the Number of Iteration(s) and $i_s$ is the Number of shifting of the Initial Point(s).] Give an initial point $(x^{(0)},\lambda_0) \in \mathcal{F}_(m) \times \{1\}.$ Set $\eta_1=10^{-12}, \eta_2=10^{-8}, c_0=50, m_0=25.$ 

$\kappa_1=\sqrt{2}, \kappa_2=9000, $ where the Step-length is determined by $\kappa_1^k, k \in Z$ and the limit of the maximum step-length is maintained by $\kappa_1^k \leq \kappa_2.$

$\epsilon_1=10^{-9}, \epsilon_2=10^{-6}.$ These are real numbers, used as thresholds for $\lambda.$  If $\lambda$ achieves a value $0\leq \lambda \leq \epsilon_1,$ then the algorithm stops with an Acceptable Solution. But, due to a stuck out for some specific reasons, if $\lambda$ achieves a value, such that, $\epsilon_1< \lambda \leq \epsilon_2,$ the algorithm stops, declaring that point as Probable Solution.\\
.
\textbf{Step 1:}  Set $\left[\begin{array}{c} 
x\\
t\\
\end{array}\right]=\left[\begin{array}{c} 
x^{(0)}\\
1\\
\end{array}\right].$ Now calculate the constant $d^{(0)}=\det  (\frac{\partial H}{\partial x}(x^{(0)},\lambda_0)). $ If $|d^{(0)}|\leq \epsilon,$ then stop else go to  step $2$.[$\epsilon \to 0,$ a threshold.]

\textbf{Step 2:} Set $c_1=c_2=0.$ Now calculate the constant $d=\det  (\frac{\partial H}{\partial x}(x,\lambda)). $ If $|d|\leq \epsilon,$ then stop else go to  step $3$.[$\epsilon \to 0,$ a threshold.]

\textbf{Step 3:} Determine the unit predictor direction $\tau^{(n)}$ by the following method:
If sign$(d)=-\text{sign}(d_0),$ then set $t_d=1-\lambda,$ else set $t_d=-\lambda.$ Calculate $w_d=-t_d(\frac{\partial H}{\partial x}(x,\lambda))^{-1}(\frac{\partial H}{\partial \lambda}(x,\lambda)),$ $\tau^{(n)}=$$\left[\begin{array}{c} 
	x_n\\
	t_n\\
\end{array}\right]=\frac{1}{\|x_d,t_d\|}\left[\begin{array}{c} 
x_d\\
t_d\\
\end{array}\right],$ $\tau=\dfrac{|t_d|}{\|x_d,t_d\|},$ where $\|x_d,t_d\|=\sqrt{x_d^2+t_d^2}.$
If $\tau \leq \eta_1,$ then set $c1=c1+1$ else reset $c1=0.$
 If $c1<c0,$ then go to step $4$ else, if $t_d\leq \epsilon_2$ then, stop with a Probable Solution else, stop due to Non-Convergence.

\textbf{Step 4:} Choosing step length: Set $ k=0; \gamma=[\nabla\mu(x)]^tx_n,$ where $\mu: R^n \to R,$   is used to increase step length in the Descent Direction(s). Set
this Function, such that, $\mu(\bar{x})\leq \mu(x),$ and $x, \bar{x} \in \mathcal{F}_(m),$ where $\bar{x}$ is the solution of the problem. $\mu(x)$ is taken as $[H_0(x)]^t[H_0(x)],$ where $H_0(x)=\left[\begin{array}{c} 
(y-w_1+v_1e+J_f^t(z-w_2+v_2e)) \\
W_1z\\
W_2y\\
y-f(z)\\
(m-\overset{n}{\underset{i=1}{\sum}}(z+w_1)_i-v_2)v_1\\
(m-\overset{n}{\underset{i=1}{\sum}}(y+w_2)_i-v_1)v_2\\
\end{array}\right].$ 

If $\gamma \geq 0, $ $x+\kappa_1^{k+1}x_n \in \mathcal{F}_{(m)}, 0<t+\kappa_1^{k+1}t_n<1, $ then set $k=k+1$ and go to step $5.$

else if $\gamma< 0, \mu(x+\kappa_1^{k+1}x_n)<\mu(x+\kappa_1^{k}x_n), x+\kappa_1^{k+1}x_n \in \mathcal{F}_(m),0<t+\kappa_1^{k+1}t_n<1, $ then set $k=k+1,$ and go to step $5.$

else reset $c_2=0,$ and jump to step $6.$

  \textbf{Step 5:} If $\kappa_1^k>\kappa_2,$ then set $k=k-1,$ $c_2=c_2+1$ and go to step $6$, else go to step $4$.
  
  \textbf{Step 6:} If $c_2<c_0,$ then go to step $7$, else if $t_n\leq \epsilon_2,$ then stop with probable solution else stop.
  
 \textbf{Step 7:}  Compute the predictor and corrector point:
 $\left[\begin{array}{c} 
 x_p\\
 t_p\\
 \end{array}\right]= \left[\begin{array}{c} 
 x\\
 t\\
 \end{array}\right]+\kappa_1^k\left[\begin{array}{c} 
 x_n\\
 t_n\\
 \end{array}\right],$  $\left[\begin{array}{c} 
 	\bar{x}_p\\
 	\bar{t}_p\\
 \end{array}\right]= \left[\begin{array}{c} 
 	x_p\\
 	t_p\\
 \end{array}\right]-[J_H(x_p,t_p)^{+}H(x_p,t_p)],$ where $[J_H(x_p,t_p)^{+}]$ is the Moore-Penrose Inverse. Now compute $\left[\begin{array}{c} 
 	\tilde{x}_p\\
 	\tilde{t}_p\\
 \end{array}\right]=$ $ \left[\begin{array}{c} 
 	x_p\\
 	t_p\\
 \end{array}\right]-2[(J_H(x_p,t_p)+J_H(\bar{x}_p,\bar{t}_p))^{+}H(x_p,t_p)]$ and $r=\|H(x_c,t_c)\|.$ Then compute the next iteration $\left[\begin{array}{c} 	x_{cc}\\
 	t_{cc}\\
 \end{array}\right]=\left[\begin{array}{c} 
 	\bar{x}_p\\
 	\bar{t}_p\\
 \end{array}\right]-2[J_H(\bar{x}_p,\bar{t}_p)+J_H(\tilde{x}_p,\tilde{t}_p)]^{+}H(\bar{x}_p,\bar{t}_p).$ Then $\left[\begin{array}{c} 
 	x_b\\
 	t_b\\
 \end{array}\right]=\left[\begin{array}{c} 
 x_{cc}\\
 	t_{cc}\\
 \end{array}\right]-J_H(x_{cc},t_{cc})^{+}H(x_{cc},t_{cc}).$ Repeating $m_0$ times get the next iteration $\left[\begin{array}{c} 
 	x_c\\
 	t_c\\
 \end{array}\right]=\left[\begin{array}{c} 
 	x_b\\
 	t_b\\
 \end{array}\right].$ If $r\leq 1,0<t_c<1,$ and $x_c \in \mathcal{F}_(m),$ then jump to step $10$ else set $k=k-1$ and go to step $8$.
 
  \textbf{Step 8:} Calculate $a = \text{min}(\kappa_1^k,\|x-x_c\|).$ If $a\leq \eta_2,$ then go to step $9$ else jump back to step $5.$
  
   \textbf{Step 9:} If $t_c \leq \epsilon_2,$ then stop with a Probable Solution else, set $ i_s = i_s + 1$ and
   jump back to Step $1$, after changing the Initial Point as, $x^{(0)} = x_c.$
   
    \textbf{Step 10:} Set $\left[\begin{array}{c} 
    	x \\
    	t \\
   \end{array}\right]=$
 $\left[\begin{array}{c} 
   x_c\\
   t_c\\
  \end{array}\right].$ 
    If $t_c\leq \epsilon_1,$ then stop with acceptable homotopy solution else set $i=i+1$ and go to step $2.$\\
     
 Note that $J_H(x,t)^{+}$ is the moore penrose inverse of the jacobian matrix $J_H(x,t).$ That is $J_H(x,t)^{+}=J_H(x,t)^T (J_H(x,t)J_H(x,t)^T)^{-1}.$
The proposed homotopy continuation method solves homotopy function by solving the initial value problem  with the following iterative process 
	    	    		 $J_j=$ $\left[\begin{array}{c} 
 x_p\\
 t_p\\
 \end{array}\right]= \left[\begin{array}{c} 
 x\\
 t\\
 \end{array}\right]+\kappa_1^k\left[\begin{array}{c} 
 x_n\\
 t_n\\
 \end{array}\right],$\\  $T_j=[J_H(x_p,t_p)^{+}H(x_p,t_p)],$\\
 $S_j=J_j-K_j= $ $\left[\begin{array}{c} 
 	\bar{x}_p\\
 	\bar{t}_p\\
 \end{array}\right]$,\\  $TT_j=$ $\left[\begin{array}{c} 
 	x_p\\
 	t_p\\
 \end{array}\right]-2[(J_H(x_p,t_p)+J_H(\bar{x}_p,\bar{t}_p))^{+}H(x_p,t_p)]=\left[\begin{array}{c} 
 	\tilde{x}_p\\
 	\tilde{t}_p\\
 \end{array}\right]$.\\ 
 $SS_j=$ $\left[\begin{array}{c} 	x_{cc}\\
 	t_{cc}\\
 \end{array}\right]=\left[\begin{array}{c} 
 	\bar{x}_p\\
 	\bar{t}_p\\
 \end{array}\right]-2[J_H(\bar{x}_p,\bar{t}_p)+J_H(\tilde{x}_p,\tilde{t}_p)]^{+}H(\bar{x}_p,\bar{t}_p),$\\			By this iterative process the proposed homotopy function achieves the order of convergence as $7^m -1.$
	    		
	    		\begin{theorem}
Suppose that the homotopy function has derivative, which is lipschitz continuous in a convex neighbourhood $\cal N$ of $c$ where $c$ is the solution of the homotopy function $H(u,t)=0,$ whose Jacobian matrix is continuous and nonsingular and bounded on $\cal N.$ 
Then the homotopy continuation method has order $7^{m}-1.$
\end{theorem}
\begin{proof}
By the Implicit Function Theorem ensures the
existence of a unique continuous solution $z(h) \in \cal N$ of $\Dot{z}(h)=-\tilde{J}^{-1}\tilde{f},$  $z(0)=u$ and $h\in(-\delta,\delta),$ for some $\delta >0.$ Define $\beta _j=\|z(h)-I_j(u,h)\|.$ From lemma \ref{coorder} $\beta_j=O(h^{7^j}).$ Then $\beta _{j+1}=\|z(h)-I_{j+1}\| \leq K{\beta _j}^7.$ Hence $\beta _{j+1}=O(h^{7^{j+1}}).$ By induction method the modified homotopy continuation method has convergency of order $7^{m}-1$
\end{proof}

\begin{theorem}
	If the homotopy curve $\Gamma_x^{(0)}$ is smooth, then the positive tangent direction $\tau^{(0)}$ at the initial point $x^{(0)}$ satisfies sign($\det \left[\begin{array}{c}
	\frac{\partial H}{\partial x \partial \lambda}(x^{(0)},1)\\
	\tau^{(0)^t}\\
	\end{array}\right]$)$<0.$ 
\end{theorem}
\begin{proof}
	From equation \ref{homf} we have
	 $ H(x,x^{(0)},\lambda)=\\$
	 $\left[\begin{array}{c} 
	(1-\lambda)(y-w_1+v_1e+J_f^t(z-w_2+v_2e))+\lambda(z-z^{(0)}) \\
	W_1z-\lambda W_1^{(0)}z^{(0)}\\
	W_2y-\lambda W_2^{(0)}y^{(0)}\\
	y-(1-\lambda)f(z)-\lambda(y^{(0)})\\
	(m-\overset{n}{\underset{i=1}{\sum}}(z+w_1)_i-v_2)v_1- \lambda((m-\overset{n}{\underset{i=1}{\sum}}(z^{(0)}+w_1^{(0)})_i-v_2^{(0)})v_1^{(0)})\\
	(m-\overset{n}{\underset{i=1}{\sum}}(y+w_2)_i-v_1)v_2- \lambda((m-\overset{n}{\underset{i=1}{\sum}}(y^{(0)}+w_2^{(0)})_i-v_1^{(0)})v_2^{(0)})\\
	\end{array}\right]=0.$\\
	 Now at the point $(x=x^{(0)},\lambda=1)$ the  value of the partial derivative is
	$\frac{\partial H}{\partial x \partial \lambda}(x,\lambda)=\begin{bmatrix}
	K_5 & K_6
	\end{bmatrix},$ where\\
	$K_5=\begin{bmatrix} M' & N'\\ \end{bmatrix},$  
	 $M'=\begin{bmatrix} I & 0 & 0 & 0 \\
	W_1^{(0)} & 0 & Z^{(0)} & 0 \\
	0 &  W_2^{(0)} & 0 & Y^{(0)} \\ 
	0 & I & 0 & 0 \\
	-v_1^{(0)} e^t & 0 & -v_1^{(0)} e^t & 0 \\
	0 & -v_2^{(0)} e^t  & 0 & -v_2^{(0)} e^t \\ \end{bmatrix}$\\ and $N'=\begin{bmatrix}0 & 0\\
	0 & 0 \\
	0  & 0 \\
	0 & 0\\
	(m-\overset{n}{\underset{i=1}{\sum}}(z^{(0)}+w_1^{(0)})_i-v_2^{(0)}) &  -v_1^{(0)} \\
	-v_2^{(0)} & (m-\overset{n}{\underset{i=1}{\sum}}(y^{(0)}+w_2^{(0)})_i-v_1^{(0)})\\
	\end{bmatrix}.$\\ $K_6=\begin{bmatrix}
	A\\
	B\\
	C\\
	D\\
	E\\
	F\\
	\end{bmatrix},$$Y^{(0)}=\text{diag}(y^{(0)}), Z^{(0)}=\text{diag}(z^{(0)})$, $W_1^{(0)}=\text{diag}(w_1^{(0)})$, $W_2^{(0)}=\text{diag}(w_2^{(0)}),\\A=-[y^{(0)}-w_1 ^{(0)}+v_1 ^{(0)}e+J_{(f^{(0)})}^t(z^{(0)}-w_2 ^{(0)}+v_2 ^{(0)}e)],\\ B=-W_1^{(0)}z^{(0)}, C=-W_2^{(0)}y^{(0)}, D= f(z^{(0)})-y^{(0)}, E=-(m-\overset{n}{\underset{i=1}{\sum}}(z^{(0)}+w_1^{(0)})_i-v_2^{(0)})v_1^{(0)}, F=-(m-\overset{n}{\underset{i=1}{\sum}}(y^{(0)}+w_2^{(0)})_i-v_1^{(0)})v_2^{(0)}. $\\
	Let positive tangent direction be $\tau^{(0)}=\left[\begin{array}{c}
	t \\ -1
	\end{array}\right]=\left[\begin{array}{c}
	(R^{(0)}_1)^{(-1)}R_2^{(0)} \\ -1
	\end{array}\right],$\\ where $R^{(0)}_1=\left[\begin{array}{cc} 
	P' & Q'\\
	\end{array}\right],$
$P'=\left[\begin{array}{cccc} 
I & 0 & 0 & 0 \\
W_1^{(0)} & 0 & Z^{(0)} & 0 \\
0 &  W_2^{(0)} & 0 & Y^{(0)} \\ 
0 & I & 0 & 0 \\
-v_1^{(0)} e^t & 0 & -v_1^{(0)} e^t & 0  \\
0 & -v_2^{(0)} e^t  & 0 & -v_2^{(0)} e^t \\
	\end{array}\right],$\\ $Q'=\left[\begin{array}{cc}
	 0 & 0\\
	  0 & 0\\
	  0  & 0\\
	  (m-\overset{n}{\underset{i=1}{\sum}}(z^{(0)}+w_1^{(0)})_i-v_2^{(0)}) &  -v_1^{(0)}\\
	   -v_2^{(0)} & (m-\overset{n}{\underset{i=1}{\sum}}(y^{(0)}+w_2^{(0)})_i-v_1^{(0)})\\ \end{array}\right]$\\
	  and $R^{(0)}_2=\left[\begin{array}{c} 
	A\\
	B\\
	C\\
	D\\
	E\\
	F\\	\end{array}\right],$ where $A=-[y^{(0)}-w_1 ^{(0)}+v_1 ^{(0)}e+J_{(f^{(0)})}^t(z^{(0)}-w_2 ^{(0)}+v_2 ^{(0)}e)],\\ B=-W_1^{(0)}z^{(0)}, C=-W_2^{(0)}y^{(0)}, D= f(z^{(0)})-y^{(0)},\\ E=-(m-\overset{n}{\underset{i=1}{\sum}}(z^{(0)}+w_1^{(0)})_i-v_2^{(0)})v_1^{(0)}, F=-(m-\overset{n}{\underset{i=1}{\sum}}(y^{(0)}+w_2^{(0)})_i-v_1^{(0)})v_2^{(0)}. $ \\
	Here $\text{det}(R^{(0)}_1)=$\\$(\frac{(m-\overset{n}{\underset{i=1}{\sum}}(z^{(0)}+w_1^{(0)})_i-v_2^{(0)})(m-\overset{n}{\underset{i=1}{\sum}}(y^{(0)}+w_2^{(0)})_i-v_1^{(0)})}{v_1^{(0)}}-v_2^{(0)})(m-\overset{n}{\underset{i=1}{\sum}}(y^{(0)}+w_2^{(0)})_i-\\v_1^{(0)})\prod_{i=1}^{n} z_i^{(0)}y_i^{(0)}\neq 0.$  Therefore, \\
	$\det\left[\begin{array}{c}
	\frac{\partial H}{\partial x \partial \lambda}(y^{(0)},1)\\
	\tau ^{(0)^t}\\
	\end{array}\right]$$=\det\left[\begin{array}{cc}
	R^{(0)}_1 & R^{(0)}_2\\
	(R^{(0)}_2)^t(R^{(0)}_1)^{(-t)} & -1\\	
	\end{array}\right]$ \\ $= \det\left[\begin{array}{cc}
	R^{(0)}_1 & R^{(0)}_2\\
	0 & -1-(R^{(0)}_2)^t(R^{(0)}_1)^{(-t)}(R^{(0)}_1)^{(-1)}R_2^{(0)} \\	
	\end{array}\right] \\ =\det(R^{(0)}_1)\det(-1-(R^{(0)}_2)^t(R^{(0)}_1)^{(-t)}(R^{(0)}_1)^{(-1)}R_2^{(0)}) \\ =-\det(R^{(0)}_1)\det(1+(R^{(0)}_2)^t(R^{(0)}_1)^{(-t)}(R^{(0)}_1)^{(-1)}R_2^{(0)}) \\ =-(\frac{(m-\overset{n}{\underset{i=1}{\sum}}(z^{(0)}+w_1^{(0)})_i-v_2^{(0)})(m-\overset{n}{\underset{i=1}{\sum}}(y^{(0)}+w_2^{(0)})_i-v_1^{(0)})}{v_1^{(0)}}-v_2^{(0)})(m-\overset{n}{\underset{i=1}{\sum}}(y^{(0)}+w_2^{(0)})_i-\\v_1^{(0)})\prod_{i=1}^{n} z_i^{(0)}y_i^{(0)}\det(1+(R^{(0)}_2)^t(R^{(0)}_1)^{(-t)}(R^{(0)}_1)^{(-1)}R_2^{(0)})<0. $ 
\end{proof}

\newpage

\section{Numerical Example}\label{oliex}
 We consider the nonlinear complementarity form of the oligopolistic market equilibrium problem and determine the equilibrium point with homotopy method.
  
Here we consider the oligopolistic market equilibrium problem operating under the Nash equilibrium concept of noncooperative behaviour, as a
 problem of the game theory.
 
To illustrate the use of the above three algorithms presented in the previous section, a numerical example\cite{articleHarker} is presented here. Consider an oligopoly with five firms, each with a total cost function of the form:
\begin{equation}\label{olix}
c_i(Q_i)=n_iQ_i+\frac{\beta_i}{\beta_i+1}{L_i}^{-\frac{1}{\beta_i}}{Q_i}^{\frac{\beta_i+1}{\beta_i}}
\end{equation}
The demand curve is given by:
\begin{equation}
\tilde{Q}=5000P^{-1.1}, \ \   P(\tilde{Q})=5000^{1/1.1}\tilde{Q}^{-1/1.1}.
\end{equation}
The parameters of the equation \ref{olix} for the five firms is given below: 
\begin{table}[ht]
	\caption{Value of parameters for five firms} 
	\centering
	\begin{tabular}{c c c c} 
		\hline\hline
		firm $i$ & $n_i$ & $L_i$ & $\beta_i$ \\ [0.5ex] 
		\hline
		1 & 10 & 5 & 1.2 \\ 
		2 & 8 & 5 & 1.1 \\
		3 & 6 & 5 & 1 \\
		4 & 4 & 5 & 0.8 \\
		5 & 2 & 5 & 0.6 \\ [1ex] 
		\hline
	\end{tabular}
\end{table}\\
\\
\\
\\

 Now we solve this problem using the above algorithm.

 	  	 To solve this problem using the homotopy method \ref{homf} with real valued parameter $\lambda,$ we first take the initial point  $z^{(0)}=$$\left[\begin{array}{c} 
 	1\\
 	1\\
 	1\\
 	1\\
 	1\\
 	\end{array}\right],$ $y^{(0)}=$$\left[\begin{array}{c} 
 	1\\
 	1\\
 	1\\
 	1\\
 	1\\
 	\end{array}\right],$ ${w_1}^{(0)}=$$\left[\begin{array}{c} 
 	1\\
 	1\\
 	1\\
 	1\\
 	1\\
 	\end{array}\right],$ ${w_2}^{(0)}=$$\left[\begin{array}{c} 
 	1\\
 	1\\
 	1\\
 	1\\
 	1\\
 	\end{array}\right],$ $v_1=0.001,$ $v_2=0.001,$ $\lambda_0=1.$ After $23$ iterations we get the result $\bar{z}=$$\left[\begin{array}{c} 
 	15.42931\\
 	12.49858\\
 	9.663473\\
 	7.165094\\
 	5.132566\\
 	\end{array}\right],$ $\bar{y}=$$\left[\begin{array}{c} 
 	0\\
 	0\\
 	0\\
 	0\\
 	0\\
 	\end{array}\right],$ $\bar{w_1}=$$\left[\begin{array}{c} 
 	0\\
 	0\\
 	0\\
 	0\\
 	0\\
 	\end{array}\right],$ $\bar{w_2}=$$\left[\begin{array}{c} 
 	15.42931\\
 	12.49858\\
 	9.663473\\
 	7.165094\\
 	5.132566\\
 	\end{array}\right],$ $\bar{v_1}=0,$ $\bar{v_2}=0,$ $\bar{\lambda}=0.$

 \begin{figure}[h]
		\centering
		\includegraphics[height=1.5in, width=4in]{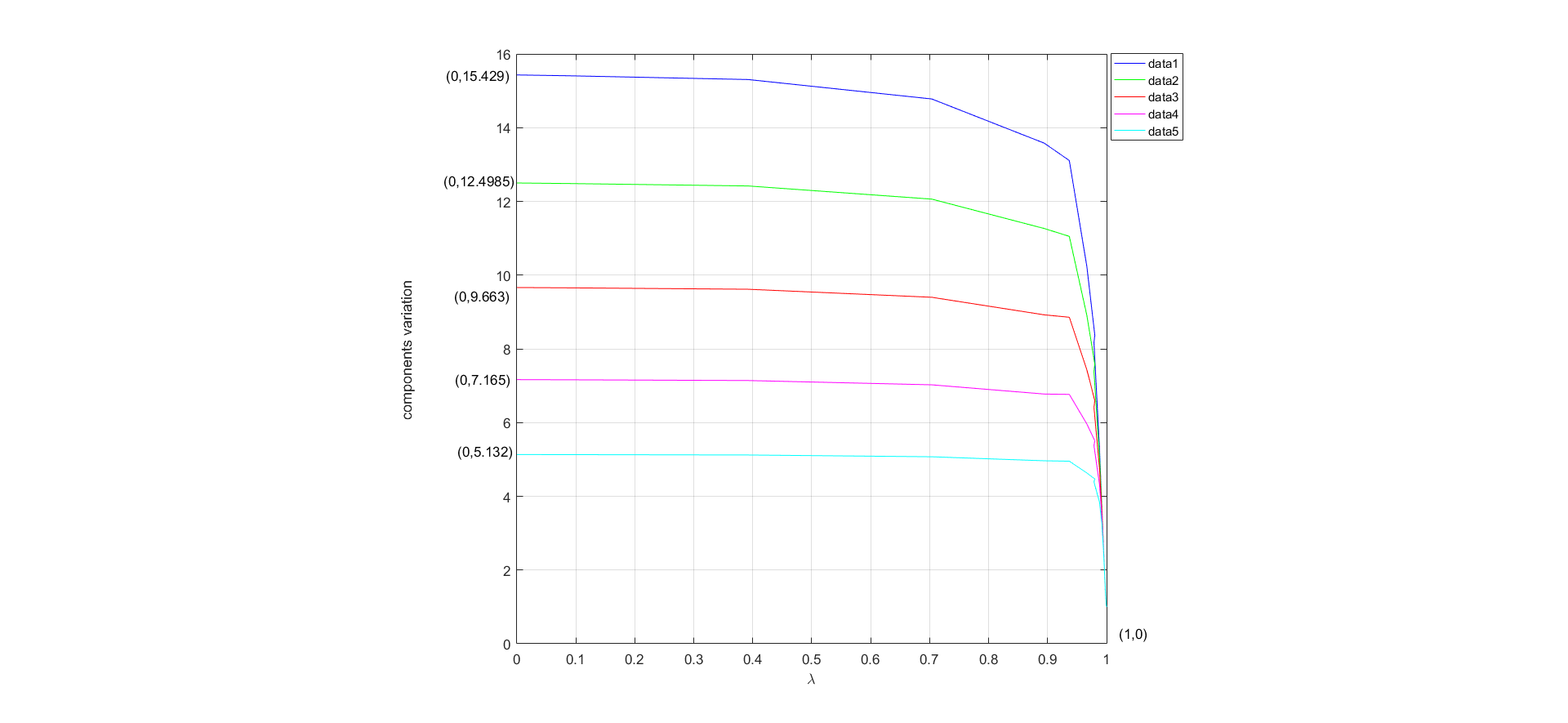}  
		\caption{Solution Approach of Oligopoly Problem}
		\label{fig0}
	\end{figure}

\bigskip

	\newpage
  \section{Conclusion}
In this study, we consider homotopy path to solve nonlinear complementarity problem based on our newly introduced homotopy function by modified homotopy continuation method.  We find the positive tangent direction of the homotopy path. We prove that the smooth curve for the proposed homotopy function is bounded and convergent under some conditions related to initial points. An oligopoly equilibrium problem is numerically solved by the proposed modified homotopy continuation method. .
\section{Acknowledgment}
The author A. Dutta is thankful to the Department of Science and Technology, Govt. of India, INSPIRE Fellowship Scheme for financial support. 
\vsp

\bibliographystyle{plain}
\bibliography{ref1}

\end{document}